\documentclass[a4paper, 10pt]{amsproc}

\usepackage{defs_opt}
\usepackage{orcidlink}
\usepackage{amsaddr}

\title{A numerical algorithm for attaining the {C}hebyshev bound in optimal learning}

\author[P. Paruchuri and D. Chatterjee]{Pradyumna Paruchuri\,\orcidlink{0000-0002-8598-5069} and Debasish Chatterjee\,\orcidlink{0000-0002-1718-653X}}
\address{%
\textsf{Systems \& Control Engineering\\ IIT Bombay, Powai\\ Mumbai 400076, India}
\\\ \\
\textsf{urls:}
\begin{minipage}{0.55\textwidth}
	\url{https://www.sc.iitb.ac.in/~pradyumn}\\
	\url{https://www.sc.iitb.ac.in/~chatterjee}
\end{minipage}
\\\ \\
\textsf{emails:} \texttt{pradyumn@sc.iitb.ac.in, dchatter@iitb.ac.in}
}
\date{\DTMnow}
\thanks{The first author was supported by the PMRF grant \textsf{RSPMRF0262} from the Government of India. The authors are in the process of submitting a patent application based on the results reported herein.}

\begin{document}

\begin{abstract}
    Given a compact subset of a Banach space, the Chebyshev center problem consists of finding a minimal circumscribing ball containing the set. In this article we establish a numerically tractable algorithm for solving the Chebyshev center problem in the context of optimal learning from a finite set of data points. For a hypothesis space realized as a compact but not necessarily convex subset of a finite-dimensional subspace of some underlying Banach space, this algorithm computes the Chebyshev radius and the Chebyshev center of the hypothesis space, thereby solving the problem of optimal recovery of functions from data. The algorithm itself is based on, and significantly extends, recent results for near-optimal solutions of convex semi-infinite problems by means of targeted sampling, and it is of independent interest. Several examples of numerical computations of Chebyshev centers are included in order to illustrate the effectiveness of the algorithm.
\end{abstract}

\keywords{optimal learning, optimal interpolation, Chebyshev center problem, convex semi-infinite programs}

\maketitle

\section{Introduction}
\label{sec:intro}

Learning \`a la approximation theory dates back at least to \cite{ref:CucSma-02}, and today it occupies the centerstage of the vibrant area of machine learning. The central idea herein is to leverage quantitative estimates germane to the field of approximation theory in the context of \emph{function learning} from (possibly) finitely many input/output data points. This function learning is realized in the form of the selection of a function from a reasonable model class (also called hypothesis space) dictated by the physics of the problem or an educated guess, that not only (nearly) justifies the data points in a certain precise sense, but is also capable of generalization beyond the given data set. Naturally, the procedure for the \emph{selection} of such a function is of central importance in terms of both applicability and numerical tractability.

Preceding works \cite{ref:MicRiv-77, ref:DeVPetWoj-17, ref:BinBonDeVPet-22} from the closely allied areas of interpolation theory and approximation theory proposed the framework of \emph{optimal recovery} in the context of function learning, wherein the aforementioned selection problem is posed in terms of furnishing a minimizer of the worst case error incurred by such a selection in the hypothesized model class. Mathematically, the preceding desideratum translates to the so-called Chebyshev center problem \cite[Chapter 15]{ref:AliTsa-21}, namely the best approximation of a set by a singleton. Let us briefly recall that in a Banach space \((\Banach, \norm{\cdot})\), a \embf{Chebyshev center} of a closed and bounded subset \(\relSet \subset \Banach\) is defined as the center of a ball of smallest radius circumscribing \(\relSet\). To wit, a Chebyshev center of \(\relSet\) is an optimizer of the variational problem: 
\begin{equation}
\label{eq:Chebyshev center problem}
\begin{aligned}
    & \minimize \limits_{f \in \Banach} && \sup \limits_{g \in \relSet}\  \norm{f - g}.
\end{aligned}
\end{equation}
The optimal value of \eqref{eq:Chebyshev center problem} is the \embf{Chebyshev radius} \(\chebRadius{\relSet}\) of \(\relSet\). In general, depending on the nature of the norm \(\norm{\cdot}\), a set \(\relSet\) may have multiple Chebyshev centers; the corresponding set is denoted by \(Z(\relSet)\). The cartoon figure given below illustrates a family of Chebyshev centers \(\zeta\) for subsets of \(\R^{2}\) under the Euclidean norm and the \(\ell_{1}\)-norm; observe that in the latter case, the set \(Z(\relSet)\) of Chebyshev centers is not a singleton.
\begin{figure}[h!]
    \centering
		\includegraphics[scale=0.3]{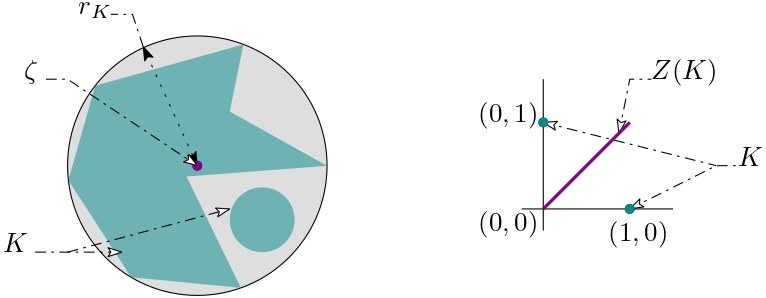}
	\caption{Illustrative examples of the Chebyshev centers and Chebyshev radius of subsets of $\R^{2}$.}
	\label{fig:Chebyshev example}
\end{figure}

In the context of learning theory, the Chebyshev center problem encodes the idea of \emph{optimal learning} in the hypothesized model class: Here \(\Banach\) represents the space of functions in which lie the hypothesis classes, and the set \(\relSet\) represents the subset of the model class of functions that satisfies the given data. The corresponding optimization problem \eqref{eq:Chebyshev center problem} is fraught with stiff numerical challenges:
\begin{itemize}[label=\textbullet, align=left, leftmargin=*]
	\item For each fixed \(f \in \Banach\), the inner maximization over \(g\) in \eqref{eq:Chebyshev center problem} is on a potentially infinite-dimensional subset of the Banach space \(\Banach\), and its solutions are rarely, if ever, expressible parametrically in closed form in \(f\).
	\item The outer minimization over \(f\) in \eqref{eq:Chebyshev center problem} is, in general, also on an infinite-dimensional Banach space \(\Banach\).
\end{itemize}
In either case, \eqref{eq:Chebyshev center problem} is numerically intractable. For reasons of computational tractability, one is, consequently, forced to ``discretize'' the various infinite-dimensional objects in \eqref{eq:Chebyshev center problem} above, and work in a finite dimensional setting;\footnote{\label{fn:finitary}Let us draw attention to the fact, as pointed out in \cite[Section 16.1]{ref:AliTsa-21}, that for computational tractability, it is imperative to restrict attention to finitary objects; consequently, considering finite-dimensional avatars of the various objects in \eqref{eq:Chebyshev center problem} is the best that can be done.} the resulting mathematical optimization is a variant of the so-called \emph{relative Chebyshev center problem}.

A \embf{relative Chebyshev center} of a closed and bounded subset \(\relSet\) with respect to a nonempty \(\Set \subset \Banach\) is given by an optimizer of:
\begin{equation}
\label{eq:relative Chebyshev center}
\begin{aligned}
    & \minimize \limits_{f \in \Set} && \sup \limits_{g \in \relSet}\  \norm{f - g},
\end{aligned}
\end{equation}
where the set \(\Set\) is chosen to be a reasonably fine ``finite'' discretization approximating the Banach space \(\Banach\) and the model class class \(\relSet\) is restricted to a suitable finite dimensional object. Nevertheless, even the resulting simplified problem \eqref{eq:relative Chebyshev center} continues to be numerically challenging:
\begin{enumerate}[label=\textup{(rCC-\arabic*)}, align=left, leftmargin=*, widest=3]
	\item \label{rcc:finite} The simplest version, although very unrealistic, is when \(\Set\) is a subset of \(\R^{n}\) and \(\relSet\) is a finite collection of points in \(\R^{n}\). The time complexity of solving such problems grows exponentially with the dimension \(n\) \cite[Chapter 15, p.\ 362]{ref:AliTsa-21} in general, and this is the current state of the art.
	\item \label{rcc:general} A more realistic setting is that of \(\Set\) being a finite dimensional subspace of \(\Banach\) and \(\relSet\) being compact (although not necessarily convex), and in this setting, the problem \eqref{eq:relative Chebyshev center} is known to be NP-hard \cite[Chapter 15, p.\ 362]{ref:AliTsa-21}. While there exist algorithms that compute the Chebyshev centers of special types of subsets \(\relSet\) of Euclidean spaces, none of them is sufficiently general to admit non-convex \(\relSet\), nor do these algorithms scale reasonably with the dimension of \(\relSet\); we refer the reader to the discussion in \cite[Chapter 15, p.\ 362]{ref:AliTsa-21} for details and references.
\end{enumerate}

\subsection*{Contributions}

\begin{enumerate}[label=\textup{(\Alph*)}, align=left, widest=B, leftmargin=*]
	\item \label{contrib:algorithm} The chief contribution of this article is a computationally tractable algorithm to solve \eqref{eq:relative Chebyshev center}.
		\begin{itemize}[label=\(\circ\), align=left, leftmargin=*]
			\item Specifically, our algorithm finds an approximant \(\approximant \in \Set\) such that \footnote{\label{fn:o-notation}Recall the Landau notation \(\phi(\genVar) = \littleOh(\genVar)\) that stands for a function \(\phi(0) = 0\) and \(\lim_{\genVar \to 0} \frac{\abs{\phi(\genVar)}}{\abs{\genVar}} = 0\).}
				\begin{equation}
					\label{eq:learning rate}
					\norm{g - \approximant} \le \chebRadius{\relSet} + \littleOh(1)\quad\text{for all } g \in \relSet,
				\end{equation}
				where the term \(\littleOh(1)\) on the right-hand side of \eqref{eq:learning rate} captures all the errors due to numerical inaccuracies, algorithmic truncation, etc., at the level of \eqref{eq:relative Chebyshev center}. Cf.\ \cite{ref:BinBonDeVPet-22} devises an algorithm to find an approximant \(\approximant \in \Set\) that satisfies
		        \[
					\text{for some constant }C > 2,\quad \norm{g - \approximant} \le C \, \chebRadius{\relSet} \quad \text{for all } g \in \relSet;
		        \]
				to the best of our knowledge, the preceding bound is currently the best available. In contrast, the algorithm reported in this article attains the \emph{best possible benchmark error bound} in the relative Chebyshev center problem \eqref{eq:relative Chebyshev center} modulo the term \(\littleOh(1)\); to wit, the constant \(C = 1\) appears in our error bound and one can do \emph{no better} than this. We refer to the situation described by \eqref{eq:learning rate} as the attainment of the \embf{Chebyshev bound} in the context of learning theory.
		
			\item Moreover, our algorithm does \emph{not} require \(\relSet\) to be convex although compactness of \(\relSet\) is critical for us. Consequently and for instance, \(\relSet\) being a finite disjoint union of non-convex compact sets is perfectly admissible in our setting, and the Chebyshev bound \eqref{eq:learning rate} continues to hold. Nor do we insist that the underlying norm should be strictly convex.
			\item Furthermore, the memory requirement of our algorithm scales \emph{linearly} with respect to the dimension of \(\relSet\); see also the discussion and references in \ref{rcc:finite} and \ref{rcc:general} above. It is conceivably possible to employ other robust optimization tools (employing, e.g., random sampling techniques) to arrive at solutions to \eqref{eq:relative Chebyshev center}, but to the best of our knowledge, no other numerically algorithm is capable of scaling linearly with respect to the dimension of \(\relSet\).
		\end{itemize}
		Naturally, our algorithm applies to the problem of optimal recovery \`a la \cite{ref:MicRiv-77} of functions from sampled measurements (which is of key relevance in signal processing), and in a sense conclusively answers the quest for a tractable numerical algorithm for optimal recovery.
	\item \label{contrib:csip} In the process of devising the aforementioned algorithm, we solve a more general problem of independent interest. This contribution consists of a numerically tractable algorithmic mechanism to solve a broad class of convex semi-infinite programs that subsumes the relative Chebyshev center problem \eqref{eq:relative Chebyshev center}. The mechanism is an extension of the algorithm recently reported in \cite{ref:DasAraCheCha-22}, making it applicable to a wider class of convex semi-infinite programs (SIPs) and also enabling it to extract optimizers of such convex SIPs. In the context of the problem \eqref{eq:relative Chebyshev center}, these features contribute to the extraction of relative Chebyshev centers of potentially non-convex (but compact) sets despite the absence of strict convexity of the underlying norm, etc.
\end{enumerate}


\subsection*{Content and organization}

In \secref{sec:learning Chebyshev center} we formally relate the problem of learning to an appropriate relative Chebyshev center problem viewed as a convex min-max optimization. The case of learning in the setting of finite-dimensional Banach spaces is treated in detail, along with a specific application to reproducing kernel Hilbert spaces. The reformulation of min-max optimization problems into convex semi-infinite programs and the applicability of our algorithm is discussed at the end of \secref{sec:learning Chebyshev center}, thereby completing the presentation of our contribution \ref{contrib:algorithm}. The technically standalone intervening \secref{sec:MSA algo} details the algorithm to solve convex SIPs and also the process of extracting optimizers of convex SIPs via regularization, which completes the presentation of our contribution \ref{contrib:csip}. Numerical experiments are presented in \secref{sec:simulations} illustrating  our algorithm and the role of regularization in extracting the Chebyshev center(s).

\section{Optimal learning via Chebyshev centers}
\label{sec:learning Chebyshev center}
A typical setting of the learning problem is that we are given a few observations on the behaviour of a function and we are required to estimate/approximate its behaviour elsewhere. The observations/measurements are of various types: point evaluations of the function if it is known to be continuous, output of linear functionals operating on the function, etc. In general, prior knowledge about the nature of admissible functions is encoded into the learning problem by specifying a hypothesis class.

Let \((\Banach, \norm{\cdot})\) be a Banach space and \(\modelClass \subset \Banach\) be a compact subset representing the hypothesis class. Information on the object of interest \(\actFunction\) is given in terms of a nonempty input-output set \(\dataSet\). Let \(\relSet \subset \modelClass\) denote the set of all possible instantiations of \(\actFunction\) that generates the data \(\dataSet\), that is,
\[
	\relSet \Let \set{g \in \modelClass \suchthat g \text{ satisfies the data in } \dataSet}.
\]
The objective of learning is to find an \(\approximant\) that minimizes the error of approximation from all possible sources from the model class \(\modelClass\) of the given data \(\dataSet\), and mathematically, this translates to the \(\min-\max\) problem:
\begin{equation}
	\label{eq:min-max learning}
	\minimize_{f \in \Banach} \quad \sup_{g \in \relSet} \norm{f - g}.
\end{equation}
The solution(s) to \eqref{eq:min-max learning} constitute Chebyshev center(s) of the set \(\relSet\) in \(\Banach\), as mentioned in the introduction. In this article, we are concerned with the premise of noise-free data. The case of noisy data is a more involved problem and will be studied in subsequent articles.

At the level of description in \eqref{eq:min-max learning}, the learning problem is numerically intractable since it involves objects in possibly infinite-dimensional spaces. One needs to discretize/restrict the search space for the approximant \(\approximant\) to a sufficiently large finite-dimensional space, say \(\discSpace\), of dimension \(\discDim\). This restriction converts problem \eqref{eq:min-max learning} of finding a Chebyshev center to that of a relative Chebyshev center problem, namely,
\begin{equation}
	\label{eq:finite search space}
	\minimize_{f \in \discSpace} \sup_{g \in \relSet} \norm{f - g}.
\end{equation}

Although the algorithm presented in \secref{sec:MSA algo} is theoretically capable in solving \eqref{eq:finite search space}, \(\relSet\) being infinite-dimensional, it still remains numerically intractable. In order to arrive at a tractable solution, one needs to restrict attention to a finite-dimensional approximation to the model class. Alternatively, one can choose to restrict attention to those functions in the model class \(\modelClass\) that lie in a fixed finite-dimensional subspace \(\discModel\) of \(\Banach\). The choice of this reduced model class is reliant on various factors, e.g., domain specific knowledge, trade-off between computational capability and tolerance to error, etc., and is up to the designer. The finite discretization of \eqref{eq:min-max learning} is the following:
\begin{equation}
	\label{eq:finite learning}
	\minimize_{f \in \discSpace} \sup_{g \in \relSet \cap \discModel} \norm{f - g}
\end{equation}
In other words, we seek the best approximant \(\approximant\) of the functions in \(\relSet \cap \discModel\). We present an algorithm to solve \eqref{eq:finite learning} exactly (modulo numerical/convergence errors) in \secref{sec:MSA algo}. To give a sneak peek, we solve the \(\max \min \max\) problem:
\begin{equation}
	\label{eq:max-min-max}
	\maximize_{g_{1}, \ldots, g_{\modelDim+1} \in \relSet \cap \discModel} \min_{f \in \discSpace} \max \set{\norm{f - g_{i}} \suchthat i = 1, \ldots, \modelDim+1}
\end{equation}
to extract the solution of the Chebyshev center problem.

In the rest of this section we discuss the avatars of \eqref{eq:finite learning} in the special settings of measurements being driven by linear functionals on the Banach space and a reproducing kernel Hilbert space (RKHS), followed by a detailed treatment of arriving at \eqref{eq:max-min-max}.

\subsection*{Measurements from linear functionals}

Suppose that the measurements/observations in the data set \(\dataSet\) are given in terms of the outputs of finitely many linear functionals from the dual space \(\dualSpace\) of \(\Banach\). That is, given a collection \((\functional)_{i=1}^{\sampleSize} \subset \dualSpace\), \(\actFunction\) is known to satisfy
\begin{equation}
	\label{eq:measurement data}
	\functional(\actFunction) = \out \quad \text{for } i = 1, \ldots, \sampleSize.
\end{equation}
At the level of description in \eqref{eq:measurement data}, \(\actFunction\) can be any candidate in a translated subspace of \(\Banach\). The information on the hypothesis class specialises the search domain to a bounded subset, say the \(\bound\)-radius ball \(\modelClass \Let \Ball{0}{\bound} \subset \Banach\). Specifically, to ensure numerical tractability, the candidate functions need to be parametrised finitely. To that end, we restrict attention to a (sufficiently large) finite dimensional subspace \(\discModel \Let \linspan \set{\evalFunc \in \Banach \suchthat j = 1, \ldots, \modelDim} \subset \Banach\). The search space for the approximant \(\approximant\) must also be made finite dimensional and we let \(\discSpace \Let \linspan \set{\basisFunc \in \Banach \suchthat i = 1, \ldots, \discDim}\). The choices of these discretizations are up to the designer.

Representing functions in \(\discModel\) with their coefficients under the basis \((\evalFunc)_{j=1}^{\modelDim}\), the data set \(\dataSet\) is translated to a subset of \(\R^{\modelDim}\) satisfying
\begin{equation}
	\label{eq:bounded coeffs} \sum_{j=1}^{\modelDim} \conCoeffs \evalFunc \in \modelClass
\end{equation}
and
\begin{equation}
	\label{eq:data satisfaction} \sum_{j=1}^{\modelDim} \conCoeffs \functional (\evalFunc) = \out \quad \text{for } i = 1, \ldots, \sampleSize.
\end{equation}
The constraint \eqref{eq:bounded coeffs} translates to \(\conCoeffs[] \in \coeffSet \subset \R^{\modelDim}\) for an appropriately defined compact set \(\coeffSet\).
Defining
\[
	\out[] \Let \pmat{\out[1]\\ \vdots \\ \out[\sampleSize]} \in \R^{\sampleSize} \quad \text{and} \quad \measMatrix \Let \pmat{\functional[1](\evalFunc[1]) & \functional[1](\evalFunc[2]) & \ldots & \functional[1](\evalFunc[\modelDim])\\
	\vdots & \vdots & \ddots & \vdots\\
	\functional[\sampleSize](\evalFunc[1]) & \functional[\sampleSize](\evalFunc[2]) & \ldots & \functional[\sampleSize](\evalFunc[\modelDim])
	},
\]
the constraint \eqref{eq:data satisfaction} is written concisely as
\[
	\measMatrix \conCoeffs[] = \out[].
\]

With the data set \(\dataSet\) as above and using the prescribed discretization, \eqref{eq:finite learning} takes the form:
\begin{equation}
	\label{eq:linear functional measurements}
	\minimize_{\funcCoeffs[] \in \R^{\discDim}} \quad \max_{\substack{\conCoeffs[] \in \coeffSet \\ \measMatrix \conCoeffs[] = \out[]}} \norm{\sum_{i=1}^{\discDim} \funcCoeffs \basisFunc - \sum_{j=1}^{\modelDim} \conCoeffs \evalFunc}.
\end{equation}

Figure \ref{fig:learning setting} illustrates a typical scenario of the learning problem described in \eqref{eq:linear functional measurements}. The region shaded in blue represents the discretized hypothesis class \(\discModel\) and the goal in the learning setting is to find the best approximant of the region colored in red.
\begin{figure}[h!]
    \centering
		\includegraphics[scale=0.15]{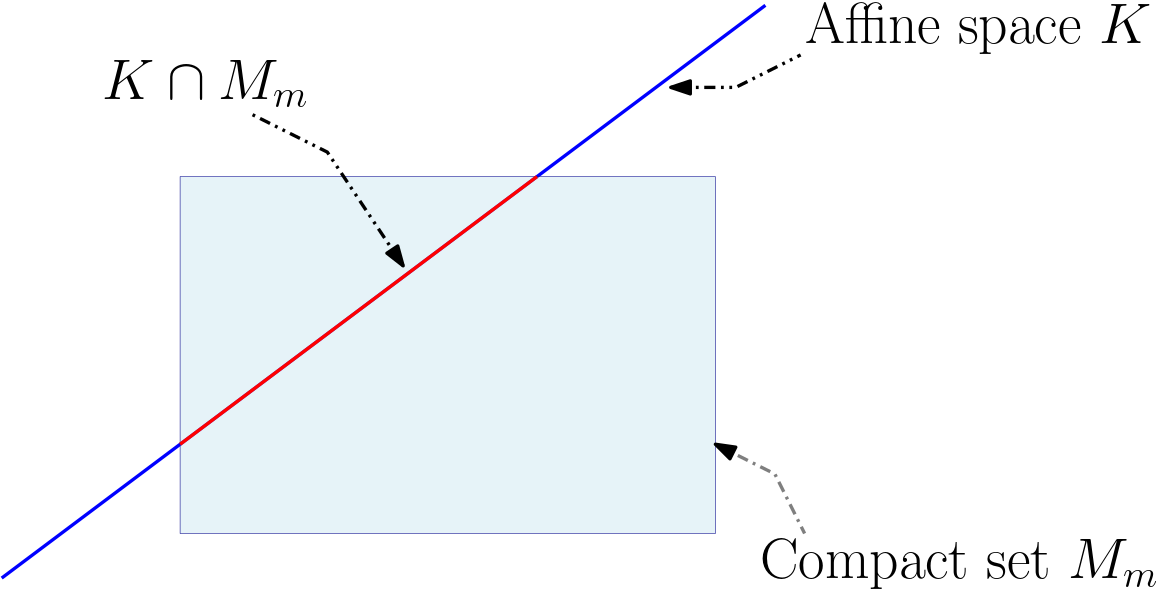}
    \caption{Depiction of the constraint sets of \eqref{eq:linear functional measurements}.}
	\label{fig:learning setting}
\end{figure}

\subsection*{RKHS setting}

    Suppose that the underlying space is a reproducing kernel Hilbert space (RKHS) denoted by \(\Hilbert\). Recall that a Hilbert space \((\Hilbert, \inprod{\cdot}{\cdot})\) of functions on a nonempty set \(\domain\) is an RKHS over the field \(\R\) if every evaluation functional is bounded, i.e., for each \(\inp[] \in \domain\) the linear map
\[
	\Hilbert \ni f \mapsto \evaluation_{\inp[]}(f) \Let f(\inp[]) \in \R
\]
is bounded. Every RKHS is equipped with its unique \emph{reproducing kernel} \(\kernel\) \cite[Chapter 2]{ref:PauRag-16}, a mapping
\[
	\kernel: \domain \times \domain \ra \R
\]
such that for every \(\inp[] \in \domain\) we have
\[
	\evaluation_{\inp[]}(f) = \inprod{\kernel(\inp[], \cdot)}{f}.
\]
Let \(\kernel_{\inp[]}\) denote the function in \(\Hilbert\) corresponding to the evaluation functional \(\evaluation_{\inp[]}\) at \(\inp[]\), that is characterized by the preceding equality. Let \(\sampleSize \in \N\) and suppose the data set \(\dataSet\) is given in terms of data points
\[
	\dataSet \Let \set{(\inp, \out) \suchthat i = 1, \ldots, \sampleSize} \subset \domain \times \R
\]
The family of functions in \(\Hilbert\) satisfying the data is described by
\[
    \relSet = \set[\Big]{f \in \Hilbert \suchthat \inprod{\kernel_{\inp}}{f} = \out \quad \text{for } i = 1, \ldots, \sampleSize}.
\]

We set the search space \(\discSpace\) for the best approximant to be the subspace (of \(\Hilbert\)) spanned by the finite family of functions (chosen by the designer)
\[
	\discSpace \Let \linspan \set{\basisFunc \in \Hilbert \suchthat i = 1, \ldots, \discDim},
\]
and let the reduced model class be restricted to a subset of the finite dimensional subspace
\[
	\discModel \Let \linspan \set{\evalFunc \in \modelClass \suchthat j = 1, \ldots, \modelDim},
\]
Note that since \(\modelClass\) is compact in \(\Hilbert\), \(\relSet \cap \discModel\) is compact. As a consequence, the coefficients \(\conCoeffs[] \in \R^{\modelDim}\) of any function in \(\relSet \cap \discModel\) are restricted to a compact subset \(\coeffSet \subset \R^{\modelDim}\) of the affine space corresponding to the data satisfaction:
\[
	\sum_{j=1}^{\modelDim} \conCoeffs \inprod{\kernel_{\inp}}{\evalFunc} = \out \quad \text{for } i = 1, \ldots, \sampleSize.
\]
Equivalently, by defining
\begin{align*}
	\out[] \Let \pmat{\out[1] \\ \vdots \\ \out[\sampleSize]} \quad \text{and} \quad
	\dataMatrix \Let  \pmat{\inprod{\kernel_{\inp[1]}}{\evalFunc[1]} & \inprod{\kernel_{\inp[1]}}{\evalFunc[2]} & \ldots & \inprod{\kernel_{\inp[1]}}{\evalFunc[\modelDim]}\\
	\vdots & \vdots & \ddots & \vdots\\
	\inprod{\kernel_{\inp[\sampleSize]}}{\evalFunc[2]} & \inprod{\kernel_{\inp[\sampleSize]}}{\evalFunc[2]} & \ldots & \inprod{\kernel_{\inp[\sampleSize]}}{\evalFunc[\modelDim]}},
\end{align*}
the constraint on the coefficients \(\conCoeffs[]\) can be written concisely as
\[
	\conCoeffs[] \in \set[\big]{w \in \coeffSet \suchthat \dataMatrix w = \out[]}
\]
Thus, in this setting, \eqref{eq:finite learning} takes the following form:
\begin{equation}
	\label{eq:RKHS learning}
		\minimize_{\funcCoeffs[] \in \R^{\discDim}} \quad \max_{\substack{\conCoeffs[] \in \coeffSet \\ \dataMatrix \conCoeffs[] = \out[]}} \norm{\sum_{i=1}^{\discDim} \funcCoeffs \basisFunc - \sum_{j=1}^{\modelDim} \conCoeffs \evalFunc}
\end{equation}

Both the problems \eqref{eq:linear functional measurements} and \eqref{eq:RKHS learning} are ready to be recast in the language of convex semi-infinite programs, which we treat next at the general level of \eqref{eq:finite learning}.

\subsection*{Reformulation of \eqref{eq:finite learning} as a convex SIP}
\label{sec:minmax problems}
We show that the min-max formulation \eqref{eq:finite learning} of the learning problem along with the special cases \eqref{eq:linear functional measurements} and \eqref{eq:RKHS learning} belong to a broader class of min-max optimization problems which can be reformulated into convex semi-infinite programs.

Consider the following min-max problem:
\begin{equation}
    \label{eq:minimax}
    \begin{aligned}
        & \minimize \limits_{\state \in \admStates} && \max \limits_{\conInp \in \admControls} \ \objective (\state, \conInp),
    \end{aligned}
\end{equation}
with the following data:
\begin{enumerate}[label=\textup{(\ref{eq:minimax}-\roman*)}, align=left, leftmargin=*, widest=iii]
    \item \label{a:closed minSet} \(\admStates \subset \R^{\sysDim}\) is a closed and convex set with nonempty interior.
    \item \label{a:compact maxSet}\(\admControls\) is a compact set.
    \item \label{a:convex objective} The objective function \(\objective: \admStates \times \admControls \ra \R\) is quasi-convex in the minimizing variable, that is \(\objective(\cdot, \conInp)\) is quasi-convex for each \(\conInp \in \admControls\).
\end{enumerate}

The finitary version \eqref{eq:finite learning} of the Chebyshev center problem fits the description of \eqref{eq:minimax} where
\begin{enumerate}[label=\textup{(\ref{eq:finite learning}-\roman*)}, align=left, leftmargin=*, widest=iii]
	\item \(\discSpace \equiv \R^{\discDim}\), which is clearly convex and closed, plays the role of \(\admStates\).
	\item The compact set \(\relSet \cap \discModel\) represents \(\admControls\).
	\item The objective/cost in \eqref{eq:finite learning} is \(\norm{f - g}\) with minimization on \(f \in \discSpace\) and maximization on \(g \in \relSet \cap \discModel\). It is clear to see that for a fixed \(g \in \relSet \cap \discModel\), the sublevel sets of \(\norm{\cdot - g}\) are norm balls in \(\discSpace\) and hence the objective is quasi-convex and continuous in the minimizing variable.
\end{enumerate}

The optimization problem \eqref{eq:minimax} can be recast as the convex semi-infinite program:
\begin{equation}
    \label{eq:SIP minimax}
    \begin{aligned}
        & \minimize_{\state, \slack} && \slack \\
        & \sbjto && \begin{cases}
            \objective(\state, \conInp) \le \slack \quad \text{for all } \conInp \in \admControls,\\
            (\slack, \state) \in \R \times \admStates.
        \end{cases}
    \end{aligned}
\end{equation}
Observe that under the hypotheses \ref{a:closed minSet} -- \ref{a:convex objective}, the convex SIP \eqref{eq:SIP minimax} satisfies the hypotheses \ref{a:closed admStates} -- \ref{a:compact constraint index} in \secref{sec:MSA algo}. Indeed,
\begin{itemize}[label=\textbullet, leftmargin=*, align=left]
    \item the objective in \eqref{eq:SIP minimax} is the linear (and hence convex and continuous) map:
\[
    \R \times \admStates \ni (\slack, \state) \mapsto \slack \in \R;
\]
    \item the constraint function
\[
    \R \times \admStates \times \admControls \ni (\slack, \state, \conInp) \mapsto \objective(\state, \conInp) - \slack \in \R
\]
is upper semi-continuous, in addition to
\[
    \R \times \admStates \ni (\slack, \state) \mapsto \objective(\state, \conInp) - \slack \in \R
\]
being quasi-convex for each fixed \(\conInp \in \admControls\);
    \item an interior point of \(\admStates\) coupled with a large enough \(\tilde{\slack}\) is also the interior point of the feasible set \(\set[\big]{(\slack, \state) \in \R \times \admStates \suchthat \objective(\state, \conInp) - \slack \le 0 \text{ for all } \conInp \in \admControls}\).
\end{itemize}

The MSA algorithm, discussed in \secref{sec:MSA algo} can be employed to recover the optimal value of the min-max problem \eqref{eq:minimax}. Under additional hypotheses such as \(\objective\) being bounded below or \(\admStates\) being compact, Proposition \ref{prop:extraction of SIP solution} can be used to extract a minimizer.

The following proposition guarantees the extraction of minimizer of \eqref{eq:minimax} via the MSA algorithm for a special case by showing uniqueness of the minimizer; see Remark \ref{rem:sufficient condition}. Although the result is well-known, we provide a brief proof for completeness.
\begin{proposition}
    \label{prop:uniqueness of minimax solution}
    If the hypothesis \ref{a:convex objective} is strengthened to strict quasi-convexity of \(\objective\) in the minimizing variable, there exists a unique minimizer for the min-max problem \eqref{eq:minimax}.
\end{proposition}

\begin{proof}
    Consider the equivalent SIP formulation \eqref{eq:SIP minimax} of the optimization problem \eqref{eq:minimax}. Let \(\slack \opt\) be the optimal value of \eqref{eq:minimax}. Suppose that \((\slack \opt, \optSol_{1})\) and \((\slack \opt, \optSol_{2})\) are two minimizers of \eqref{eq:SIP minimax}. Let \(\solution\) be the midpoint of \(\optSol_{1}\) and \(\optSol_{2}\). We have \(\solution \in \admStates\) due to convexity of \(\admStates\).

    For a fixed \(\conInp \in \admControls\), we have
    \[
        \objective(\optSol_{1}, \conInp) \le \slack \opt \quad \text{and} \quad \objective(\optSol_{2}, \conInp) \le \slack \opt.
    \]
    By strict quasi-convexity of \(\objective\) in the first argument,
    \[
        \objective(\solution, \conInp) < \max \set[\big]{\objective(\optSol_{1}, \conInp), \objective(\optSol_{2}, \conInp)} < \slack \opt.
    \]
    Since \(\conInp \in \admControls\) is arbitrary and \(\admControls\) is compact,
    \[
        \max \limits_{\conInp \in \admControls} \objective(\solution, \conInp) < \slack \opt.
    \]
    This contradicts the optimality of \(\slack \opt\) and this completes our proof.
\end{proof}

\subsubsection*{Chebyshev centers}
\label{s:Chebyshev center}



In the context of the optimal learning problem introduced in \secref{sec:learning Chebyshev center}, a Chebyshev center of \(\relSet\) can be equivalently defined as an optimizer of the min-max problem
\begin{equation}
    \label{eq:Chebyshev center}
    \minimize_{\state \in \Banach} \; \sup \limits_{\candidate \in \relSet} \norm{\state - \candidate};
\end{equation}
indeed, observe that the Chebyshev radius \(\chebRadius{\relSet}\) is the optimal value of \eqref{eq:Chebyshev center}. Moreover, Chebyshev centers of the compact set \(\relSet\) coincide with those of its convex hull \(\chull \relSet\). The optimization problem \eqref{eq:Chebyshev center} can be reformulated into the following convex semi-infinite program:
\begin{equation}
\label{eq:Chebyshev reform}
\begin{aligned}
    & \minimize \limits_{\slack, \state} && \slack\\
    & \sbjto && \begin{cases}
        \norm{\state - \candidate} \le \slack \quad \text{for all } \candidate \in \relSet,\\
        (\slack, \state) \in \lcro{0}{+\infty} \times \Banach,
    \end{cases}
\end{aligned}
\end{equation}
in the sense that the value of \eqref{eq:Chebyshev reform} is the Chebyshev radius of \(\relSet\) and an optimizer in \(\state\) of \eqref{eq:Chebyshev reform} is a Chebyshev center of \(\relSet\).

If the norm \(\norm{\cdot}\) on the space \(\Banach\) is strictly convex, then the objective function of \eqref{eq:Chebyshev center} is strictly quasi-convex in the minimizing variable \(\state\), and consequently, in the light of Proposition \ref{prop:uniqueness of minimax solution}, there exists a unique Chebyshev center of \(\relSet\). Otherwise, Chebyshev centers of \(\relSet\) may be extracted by means of the regularization procedure of Proposition \ref{prop:extraction of SIP solution}. To wit, the MSA algorithm and its extension in Proposition \ref{prop:extraction of SIP solution} furnishes a numerically tractable technique for the exact computation of Chebyshev centers of compact subsets of finite-dimensional normed vector spaces. We shall illustrate the technique in \secref{sec:simulations} with specific numerical examples.

\section{Extraction of solutions to convex semi-infinite programs}
\label{sec:MSA algo}
This section contains a detailed treatment of a mechanism to extract solutions -- both the optimal value and optimizers -- of convex semi-infinite programs. The results herein are of independent interest and the Chebyshev center problem (i.e., the computation of both the Chebyshev radius and Chebyshev centers) turns out to be special cases of the mechanism.

Consider the following convex semi-infinite program
\begin{equation}
    \label{eq:convex SIP}
    \begin{aligned}
        & \minimize && \objective (\state) \\
        & \sbjto && \begin{cases}
            \constraintMap(\state, \conInp) \le 0 \quad \text{for all } \conInp \in \admControls,\\
            \state \in \admStates,
        \end{cases}
    \end{aligned}
\end{equation}
with the following data:
\begin{enumerate}[label=\textup{(\ref{eq:convex SIP}-\roman*)}, align=left, widest=iii, leftmargin=*]
    \item \label{a:closed admStates} \(\admStates \subset \R^{\sysDim}\) is a closed and convex set with nonempty interior.
    \item \label{a:Slater condition} The feasible set \(\feasibleStates \Let \set[\big]{\state \in \admStates \suchthat \constraintMap (\state, \conInp) \le 0 \text{ for all } \conInp \in \admControls}\) is assumed to have nonempty interior.
    \item \label{a:objective} The objective function \(\admStates \ni \state \mapsto \objective(\state) \in \R\) is quasi-convex and upper semi-continuous.
    \item \label{a:constraint map} The constraint function \(\admStates \times \admControls \ni (\state, \conInp) \mapsto \constraintMap(\state, \conInp) \in \R\) is upper semi-continuous in both the variables and is strictly quasi-convex in \(\state\) for each fixed \(\conInp\).
    \item \label{a:compact constraint index} The constraint index set \(\admControls\) is a compact set.
\end{enumerate}

\begin{remark}
	\label{rem:infinite dimensional constraints} The set \(\admControls\) is \embf{not} required to be finite-dimensional, but for numerical tractability one is typically forced to consider finite-dimensional versions of \(\admControls\) in practice.
\end{remark}

Convex SIPs arise in a plethora of applications including portfolio optimization, statistics, learning, estimation theory, and approximation theory among others. We refer the reader to the sweeping survey \cite{ref:BerBroCar-10} and the textbooks \cite{ref:BenElGNem-09, ref:MarMarYohSal-19} for details and applications. In addition, we also point to the recent article \cite{ref:DasAraCheCha-22} for background literature and perspective; the body of results in the current section may be viewed as a natural continuation of \cite{ref:DasAraCheCha-22}.

The algorithm established in \cite{ref:DasAraCheCha-22} for solving convex semi-infinite programs via targeted sampling, which we shall call the MSA algorithm in the sequel,\footnote{The name is derived from the three students who contributed to the results in \cite{ref:DasAraCheCha-22}: Mishal Assif P.K., Souvik Das, and Ashwin Aravind.} readily gives the optimal value of a special case of \eqref{eq:convex SIP}.\footnote{The precise hypotheses of the special case will be explained below.} To the best of our knowledge, till date it is the only numerically tractable algorithm that computes the precise value of convex SIPs. However, since it solves a relaxed convex program \cite[Equation (2.7)]{ref:DasAraCheCha-22}, the set of minimizers obtained thereby is only a priori known to be a superset of the original solutions. In the case of the objective \(\objective\) being strictly convex, the solution to the relaxed program \cite[Equation (2.7)]{ref:DasAraCheCha-22} coincides with that of the original problem in the sense that
\begin{itemize}[label=\textbullet, align=left, leftmargin=*]
    \item the optimal values coincide, and
    \item the optimizer to the relaxed problem also optimizes the original SIP.
\end{itemize}

\begin{remark}
	\label{rem:sufficient condition}
	A sufficient condition for the optimizer to the relaxed program \cite[Equation (2.7)]{ref:DasAraCheCha-22} to be an optimizer to the original problem is the uniqueness of optimizers for the relaxed programs. Strict convexity of the objective \(\objective\) is one way to ensure that this sufficient condition is satisfied. It is important to note that the sufficient condition is the uniqueness of minimizers for the relaxed programs and not just for the original problem. The example in \secref{sec:solid disk} indicates this requirement.
\end{remark}

As an immediate illustration, consider the problem of constructing the Chebyshev ball of a convex subset \(\relSet\) of \(\R^2\) defined by
\[
	\relSet = \set[\bigg]{(\xpos, \ypos) \in \R^{2} \suchthat
	\begin{aligned}
				& \xpos + 2 \ypos \le 2,\; -\xpos + 2 \ypos \le 2, \\
				& -\xpos + 4 \ypos \ge -2, \; \xpos \ge -2.
	\end{aligned}}
\]
Recall that the \(\ell_1\)-Chebyshev ball of \(\relSet\) is a circumscribing \(\ell_1\)-ball of the smallest radius. The mathematical problem of finding a Chebyshev ball of \(\relSet\) may be formulated as a solution to the min-max problem
\[
	\min_{x\in\R^2} \max_{y\in \relSet} \norm{x-y}_{\ell_1},
\]
whose \emph{value} is the Chebyshev radius of \(\relSet\) and an optimizer (in the variable \(x\)) is a Chebyshev center of \(\relSet\). This min-max problem permits a reformulation as a convex SIP in a standard way, and the MSA algorithm applied to that convex SIP leads to the correct Chebyshev radius of \(2.5\) but the \(\ell_1\)-ball of radius \(2.5\) obtained from the algorithm may not circumscribe \(\relSet\), as shown in the following figure:
\[
	\includegraphics[scale=0.6]{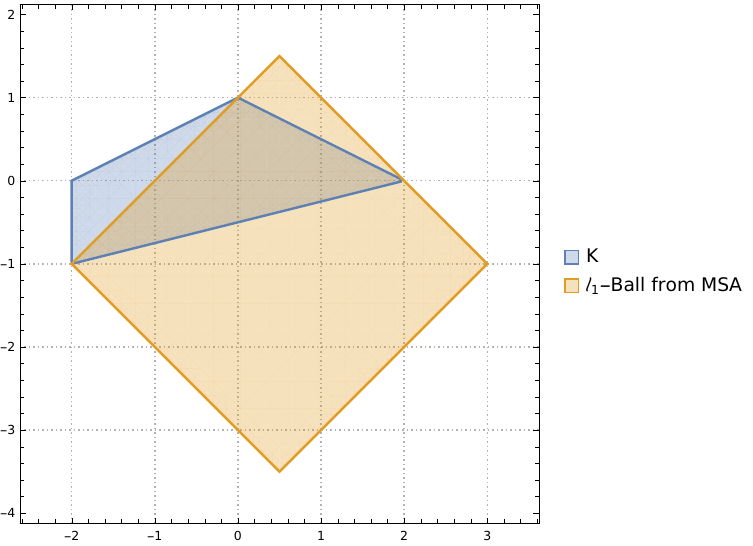}
\]
This situation arises because while the MSA algorithm was designed to match the \emph{values} of convex SIPs (and the MSA algorithm correctly finds the Chebyshev radius in this example), the optimizers of the two problems may be different.\footnote{We shall revisit this example in \secref{sec:simulations}.}

In subsection \secref{s:optimizer extraction}, we establish a mechanism to extend the capability of the MSA algorithm to extract \emph{optimizers} for general convex objective functions. In particular, our contributions are the following:
\begin{enumerate}[label=(\Roman*), align=left, leftmargin=*, widest=II]
    \item \label{clm:extension} The original MSA algorithm in \cite{ref:DasAraCheCha-22} is generalized and shown to be applicable to the data \ref{a:closed admStates} -- \ref{a:compact constraint index}. This entails the following generalizations:
        \begin{itemize}
            \item The ambit of \cite[hypothesis ((1.1)-c)]{ref:DasAraCheCha-22} involving convexity and continuity of the objective \(\objective\) is broadened to include quasi-convexity and upper semi-continuity of \(\objective\);
            \item In \cite{ref:DasAraCheCha-22}, the constraint map \(\constraintMap\) is required to be  convex in \(\state\) for each fixed \(\conInp \in \admControls\), and jointly continuous in \(\state\) and \(\conInp\). The ambit of this hypothesis is broadened to include strict quasi-convexity of \(\constraintMap\) in \(\state\) for every \(\conInp \in \admControls\) and joint upper semi-continuity in both variables.
        \end{itemize}
    \item \label{clm:regularization} Consider the data \ref{a:closed admStates} -- \ref{a:compact constraint index} associated with the problem \eqref{eq:convex SIP}. When the objective \(\objective\) is convex and continuous, we establish an approach via regularization to extract an optimizer of \eqref{eq:convex SIP} using the MSA algorithm itself.
\end{enumerate}

\subsection{Generalization of the MSA algorithm (\`a la \ref{clm:extension})}
\label{s:generalization}

We first briefly discuss the key ideas behind the MSA algorithm and point out how the same algorithm is applicable to the generalized data \ref{a:closed admStates} -- \ref{a:compact constraint index} accompanying the problem \eqref{eq:convex SIP}.

Let \(\relObjective: \admControls^{\sysDim} \ra \R\) be the map defined by
\begin{equation}
    \label{eq:finite constraint problem}
    \relObjective \bigl(\conInp_{1}, \ldots, \conInp_{\sysDim}\bigr) = \inf \limits_{\state \in \admStates} \set[\big]{\objective(\state) \suchthat \constraintMap(\state, \conInp_{i}) \le 0 \quad \text{for } i = 1, \ldots, \sysDim}.
\end{equation}
    Note that the evaluation of \(\relObjective\) involves a finitely constrained convex program which is a relaxed version of \eqref{eq:convex SIP}. Since the minimization in \eqref{eq:finite constraint problem} is over a larger set compared to that in \eqref{eq:convex SIP}, the evaluation of the function \(\relObjective\) yields a value that is at most equal to the optimal value of \eqref{eq:convex SIP}.

        In the proof of \cite[Theorem 1]{ref:DasAraCheCha-22} the authors invoke \cite[Theorem\ 4.1]{ref:Borwein-81} to show equivalence between \eqref{eq:convex SIP} under the specialized data wherein the objective \(\objective\) is stipulated to be a convex and continuous function, and the constraint map \(\constraintMap\) is required to be convex in \(\state\) and jointly continuous in \(\state\) and \(\conInp\), and the global optimization problem:
\begin{equation}
    \label{eq:equivalent problem}
    \sup \limits_{(\conInp_{1}, \ldots, \conInp_{\sysDim}) \in \admControls^{\sysDim}} \relObjective \bigl(\conInp_{1}, \ldots, \conInp_{\sysDim}\bigr).
\end{equation}
The equivalence claimed in \cite{ref:DasAraCheCha-22} is in the sense that the value of \eqref{eq:equivalent problem} is precisely the value of \eqref{eq:convex SIP} under the aforementioned specialised data. The MSA algorithm solves the global optimization problem \eqref{eq:equivalent problem} to obtain the optimal value of the convex semi-infinite program.

        \cite[Theorem\ 4.1]{ref:Borwein-81} can be invoked on the broader class of data in \ref{a:closed admStates} -- \ref{a:compact constraint index} and the proof of \cite[Theorem 1]{ref:DasAraCheCha-22} follows through as is. Hence, solving \eqref{eq:equivalent problem} is sufficient to obtain the optimal value of \eqref{eq:convex SIP} even with the data \ref{a:closed admStates} -- \ref{a:compact constraint index}; consequently, the proof of \cite[Theorem 1]{ref:DasAraCheCha-22} carries over verbatim to our more general context.

\subsection{Extracting optimizers via regularization (\`a la \ref{clm:regularization})}
\label{s:optimizer extraction}

Let \(\bigl(\conInp^{\circ}_{1}, \ldots, \conInp^{\circ}_{\sysDim}\bigr)\) be a global optimiser of \(\relObjective\) in \(\admControls^{\sysDim}\). Then \(\relObjective \bigl(\conInp^{\circ}_{1}, \ldots, \conInp^{\circ}_{\sysDim}\bigr)\) is the optimal value of \eqref{eq:convex SIP} in view of our arguments in \secref{s:generalization}. 
In addition, the optimizers of \eqref{eq:convex SIP} lie in the set of solutions to the minimization problem in \eqref{eq:finite constraint problem} that comes up while evaluating \(\relObjective\) at one of its global optimizer \(\bigl(\conInp^{\circ}_{1}, \ldots, \conInp^{\circ}_{\sysDim}\bigr)\). Since this minimization problem is on a relaxed constraint set compared to that in \eqref{eq:convex SIP}, the challenge at this stage is to extract those optimizers that lie in the feasible set of the original optimization problem \eqref{eq:convex SIP}.

Before delving into a new method of extracting optimizers (to be established below), we make a few preliminary observations on convex optimization in \secref{s:convex opt} based on which our method is built.

\subsubsection{Results from convex optimization}
\label{s:convex opt}

Consider the following convex program
\begin{equation}
\label{eq:convex problem}
\tag{\CSIP}
    \begin{aligned}
    \minimize_{\state \in \feasibleStates} && \objective(\state),
    \end{aligned}
\end{equation}
where \(\feasibleStates \subset \R^{\sysDim}\) is a closed and convex set, \(\objective: \R^{\sysDim} \ra \R\) is a convex and continuous. Let \(\optVal\) be the optimal value of \eqref{eq:convex problem} and let \(\argmin \CSIP\) denote the set of solutions of the convex program \eqref{eq:convex problem}.

Consider a variation of \eqref{eq:convex problem} where we perturb the objective by a strictly convex function \(\regularizer\), with \(\eps > 0\):
\begin{equation}
\label{eq:perturbed convex problem}
    \tag{\CSIP[\eps]}
    \begin{aligned}
        & \minimize_{\state \in \feasibleStates} && \objective(\state) + \eps \regularizer(\state) \\
    \end{aligned}
\end{equation}
Note that by construction, the problem \eqref{eq:convex problem} and the perturbed problems \eqref{eq:perturbed convex problem} have the same feasible set \(\feasibleStates\).


\begin{lemma}
    \label{lem:minimizing sequence}
    Let \(\solution[\eps]\) be the unique minimizer of \eqref{eq:perturbed convex problem}. Then \(\objective(\solution[\eps])\) is monotone non-increasing as \(\eps \downarrow 0\) and \(\regularizer(\solution[\eps])\) is monotone non-decreasing as \(\eps \downarrow 0\). Moreover,
    \[
        \inf \limits_{\eps > 0} \objective(\solution[\eps]) = \optVal \quad \text{and} \quad \sup \limits_{\eps > 0} \regularizer(\solution[\eps]) \le \inf \limits_{\state \in \argmin \CSIP} \regularizer(\solution).
    \]
\end{lemma}

\begin{proof}
    Let \(\eps > \eps' > 0\). Since \(\solution[\eps]\) is optimal for \(\CSIP[\eps]\) and  both \(\solution[\eps], \solution[\eps']\)  lie in the feasible set,
    \[
        \objective(\solution[\eps]) + \eps \regularizer(\solution[\eps]) \le \objective(\solution[\eps']) + \eps \regularizer(\solution[\eps']),
    \]
    which gives
    \[
        \objective(\solution[\eps]) - \objective(\solution[\eps']) \le \eps \bigl(\regularizer(\solution[\eps']) - \regularizer(\solution[\eps])\bigr).
    \]
    Similarly, by optimality of \(\solution[\eps']\) for \(\CSIP[\eps']\) we have
    \[
        \objective(\solution[\eps']) + \eps \regularizer(\solution[\eps']) \le \objective(\solution[\eps']) + \eps \regularizer(\solution[\eps]),
    \]
    leading to
    \[
        \objective(\solution[\eps]) - \objective(\solution[\eps']) \ge \eps' \bigl(\regularizer(\solution[\eps']) - \regularizer(\solution[\eps])\bigr).
    \]
    Combining the above two inequalities,
    \[
        (\eps - \eps') \bigl(\regularizer(\solution[\eps']) - \regularizer(\solution[\eps])\bigr) \ge 0,
    \]
    which shows that
    \[
        \regularizer(\solution[\eps']) \ge \regularizer (\solution[\eps]).
    \]
    Consequently, 
    \[
        \objective(\solution[\eps']) \le \objective(\solution[\eps]).
    \]
    Thus the family \(\set{\objective(\solution[\eps])}_{\eps>0}\) decreases with \(\eps\).

    Clearly, from the definition of \(\optVal\) we have \(\objective(\solution[\eps]) \ge \optVal\).
    For \(\state \in \argmin \CSIP\),
    \[
        \objective(\solution[\eps]) + \eps \regularizer(\solution[\eps]) \le \objective(\state) + \eps \regularizer(\state)
    \]
    which yields
    \[
        0 \le \objective(\solution[\eps]) - \optVal \le \eps \bigl (\regularizer(\state) - \regularizer(\solution[\eps])\bigr)
    \]
    and hence
    \[
        \sup \limits_{\eps > 0} \regularizer(\solution[\eps]) \le \inf \limits_{\state \in \argmin \CSIP} \regularizer (\state).
    \]
    This establishes the second assertion. Moreover,
    \[
        \objective(\solution[\eps]) \le \optVal + \eps (\regularizer(\state) - \regularizer(\solution[\eps]))
    \]
    Since \(\regularizer\) is bounded on \(\feasibleStates\), taking infimum over \(\eps > 0\) on both sides yields
    \[
        \inf \limits_{\eps > 0} \objective(\solution[\eps]) \le \optVal
    \]
    Thus \(\inf \limits_{\eps > 0} \objective(\solution[\eps]) = \optVal\), thereby establishing the first assertion, and this completes our proof.
\end{proof}


\begin{proposition}
    \label{prop:characterization of recovered solution}
    The family of solutions \(\set{\solution[\eps]}_{\eps > 0}\) has a unique cluster point \(\optSol\). Moreover, \(\optSol \in \argmin \CSIP\),
    \[
        \text{\(\optSol\) solves}\quad \minimize_{\state \in \argmin \CSIP} \regularizer(\state),
    \]
    and
    \[
        \lim \limits_{\eps \downarrow 0} \solution[\eps] = \optSol.
    \]

\end{proposition}

\begin{proof}
    Let \((\solution[\eps_{n}])_{n \in \N}\), with \(\eps_{n} \downarrow 0\), be a subsequence in \(\set{\solution[\eps]}_{\eps>0}\) converging to \(\ol{\state} \in \feasibleStates\). 
    It follows from continuity of \(\objective\) and Lemma \ref{lem:minimizing sequence} that
    \[
        \objective(\ol{\state}) = \lim \limits_{n \ra +\infty} \objective(\solution[\eps_{n}]) = \inf \limits_{n \ra +\infty} \objective(\solution[\eps_{n}]) = \optVal,
    \]
    since \((\objective(\solution[\eps_{n}]))_{n \in \N}\) is a monotone sequence. This indicates that \(\ol{\state} \in \argmin \CSIP\).


    On the one hand, by continuity of \(\regularizer\) and from Lemma \ref{lem:minimizing sequence},
    \[
        \regularizer(\ol{\state}) = \lim \limits_{n \ra +\infty} \regularizer(\solution[\eps_{n}]) = \sup \limits_{n \ra +\infty} \regularizer(\solution[\eps_{n}]) \le \inf \limits_{\state \in \argmin \CSIP} \regularizer(\state) ,
    \]
    while on the other hand, since \(\ol{\state} \in \argmin \CSIP\),
    \[
        \regularizer(\ol{\state}) \ge \inf \limits_{\state \in \argmin \CSIP} \regularizer(\state).
    \]

    Thus \(\ol{\state}\) minimizes \(\regularizer\) on \(\argmin \CSIP\). This indicates that the cluster points of \(\set{\solution[\eps]}_{\eps>0}\) solve the minimization problem
    \[
    \minimize_{\state \in \argmin \CSIP} \regularizer(\state).
    \]
    Since there exists a unique minimizer \(\optSol\) by virtue of strict convexity of \(\regularizer\) and convexity of \(\argmin \CSIP\), the cluster point is unique, say \(\optSol\).

    Since the sublevel sets \(\set{\state \in \admStates \suchthat \objective(\state) \le \alpha}\) for \(\alpha \in \R\) are bounded sets by assumption, the family \(\set{\solution[\eps]}_{\eps>0}\) is bounded. Since every subsequential limit of the family \(\set{\solution[\eps]}_{\eps>0}\) is \(\optSol\), the family \(\set{\solution[\eps]}_{\eps>0}\) itself converges to \(\optSol\).
\end{proof}

\subsubsection{Extraction of optimizers}
\label{s:extraction}

Let us consider the following more general version of \eqref{eq:convex SIP}:
\begin{equation}
\label{eq:modified convex SIP}
    \begin{aligned}
        & \minimize && \objective(\state) + \eps \regularizer(\state) \\
        & \sbjto && \begin{cases}
            \constraintMap(\state, \conInp) \le 0 \quad \text{for } \conInp \in \admControls,\\
            \state \in \admStates,
    \end{cases}
    \end{aligned}
\end{equation}
where in addition to the data \ref{a:closed admStates} -- \ref{a:compact constraint index}, we impose
\begin{enumerate}[label=\textup{(\ref{eq:convex SIP}-\roman*)}, align=left, leftmargin=*, widest=vi, start=6]
    \item the map \(\regularizer: \admStates \ra \R\) is chosen to be a positive, strictly convex function and \(\eps > 0\).
\end{enumerate}

\begin{proposition}
    \label{prop:extraction of SIP solution}
    Suppose the objective function \(\objective\) is convex, continuous and has bounded sublevel sets. Then the MSA algorithm extracts a minimizing sequence by solving \eqref{eq:modified convex SIP} for \(\eps \downarrow 0\) and the sequence converges to a solution to \eqref{eq:convex SIP}.
\end{proposition}

Proposition \ref{prop:characterization of recovered solution} enables us to construct a minimizing sequence comprised of solutions of \eqref{eq:modified convex SIP}, with \(\eps \downarrow 0\), that converges to an optimizer of \eqref{eq:convex SIP}.
\begin{proof}
    Since the objective function \(\objective\) is continuous and has bounded sublevel sets, the constraint set \(\admStates\) in \eqref{eq:convex SIP} can be replaced with the compact set \(\tilde{\admStates} \Let \set[\big]{\state \in \admStates \suchthat \objective(\state) \le \objective(\tilde{\state} + 1)}\). This guarantees that the optimal value of \eqref{eq:convex SIP} is finite and is attained. By the choice of a positive \(\regularizer\), the same argument holds true for \eqref{eq:modified convex SIP}.

    For \(\eps>0\), the function \(\objective + \eps \regularizer\) is, by construction, strictly convex and hence the convex SIP \eqref{eq:modified convex SIP} has a unique solution. In view of \cite[Proposition 2]{ref:DasAraCheCha-22}, we know that the optimizer of \eqref{eq:modified convex SIP} is obtained by the MSA algorithm, that is by solving the minimization problem in \eqref{eq:finite constraint problem} corresponding to the objective in \eqref{eq:modified convex SIP}.

    Proposition \ref{prop:characterization of recovered solution} ensures that for any sequence \(\eps \downarrow 0\), the family of solutions \(\solution[\eps]\), obtained by solving \eqref{eq:modified convex SIP} using MSA algorithm, is a minimizing sequence of the objective function \(\objective\) that converges to an optimizer of \eqref{eq:convex SIP}. 
\end{proof}

\begin{remark}
    Proposition \ref{prop:extraction of SIP solution} is applicable even in the case when the objective \(\objective\) is linear if the domain of interest \(\admStates\) is restricted to be compact.
\end{remark}


\section{Numerical experiments}
\label{sec:simulations}
This section is devoted to the illustration of the extended MSA algorithm based on Proposition \ref{prop:extraction of SIP solution}. Standard optimization routines from the \textsf{SciPy} library and standard solvers from \textsf{Mathematica 12.1} have been employed in solving the examples provided in this section. We begin with a simple example of linear optimization on a solid disk; this problem can be readily solved using quadratic solvers, but for illustration purposes we reformulate it as a convex SIP.

\subsection{Optimization on a solid disk}
\label{sec:solid disk}

Consider the following optimization problem:
\begin{equation}
    \label{eq:circle optimization}
    \begin{aligned}
        & \minimize_{\xpos, \ypos} && \ypos \\
        & \sbjto && \begin{cases}
            \xpos^{2} + \ypos^{2} \le 9, \\
            -4 \le \xpos \le 4,\\
            -4 \le \ypos \le 4.
        \end{cases}
    \end{aligned}
\end{equation}
It is easy to check that the optimal value of \eqref{eq:circle optimization} is \(-3\) and is attained at \(\pmat{\xpos \opt, \ypos \opt} = \pmat{0, -3}\). The above problem can be recast into a convex SIP with linear constraints:
\begin{equation}
    \label{eq:circle SIP}
    \begin{aligned}
        & \minimize_{\xpos, \ypos} && \ypos \\
        & \sbjto && \begin{cases}
            \xpos \cos(\theta) + \ypos \sin(\theta) \le 3 \quad \text{for all } \theta \in \lcrc{0}{2\pi},\\
            -4 \le \xpos \le 4,\\
            -4 \le \ypos \le 4.
        \end{cases}
    \end{aligned}
\end{equation}
When employed directly, the MSA algorithm selects points on the line \(\ypos = -3\), which is indicated by the green line in Figure \ref{fig:circle example}.

\begin{figure}[h!]
    \centering
    \begin{subfigure}{0.45\textwidth}
        \centering
        \includegraphics[scale=0.5]{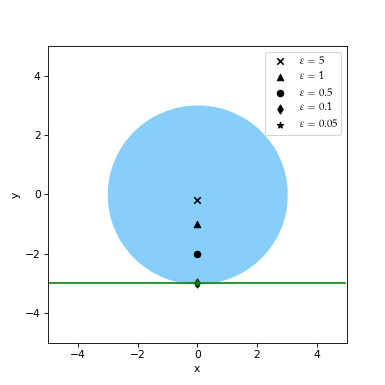}
        \subcaption{Solutions of \eqref{eq:circle SIP} regularized by \(\regularizer_{\textrm{a}}\)}
    \label{fig:circle optimization (a)}
    \end{subfigure} \qquad
    \begin{subfigure}{0.45\textwidth}
        \centering
        \includegraphics[scale=0.5]{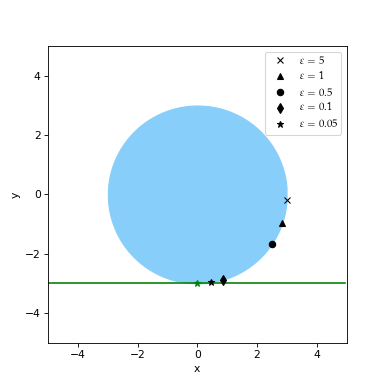}
        \subcaption{Solutions of \eqref{eq:circle SIP} regularized by \(\regularizer_{\textrm{b}}\)}
    \label{fig:circle optimization (b)}
    \end{subfigure}
	\caption{Solutions of \eqref{eq:circle SIP} extracted via regularization.}
	\label{fig:circle example}
\end{figure}

Figure \ref{fig:circle optimization (a)} shows a sequence of solutions corresponding to decreasing values of \(\eps\), of \eqref{eq:circle SIP} with the perturbation function
\[
	\R^2\ni (\xpos, \ypos)\mapsto \regularizer_{\textrm{a}}(\xpos, \ypos) \Let \tfrac{1}{2} \bigl(\xpos^{2} + \ypos^{2}\bigr),
\]
and Figure \ref{fig:circle optimization (b)} shows the corresponding sequence of solutions with the perturbation function
\[
	\R^2\ni (\xpos, \ypos)\mapsto \regularizer_{\textrm{b}}(\xpos, \ypos) \Let \tfrac{1}{2} \bigl((\xpos - 1)^{2} + \ypos^{2}\bigr).
\]

Notice that the sequence of solutions obtained is dependent on the choice of perturbation \(\regularizer\) and as described in Proposition \ref{prop:characterization of recovered solution}, the sequences may converge to different optimizers of the original SIP. However in this example, since the SIP \eqref{eq:circle SIP} has a unique solution at \(\pmat{0, -3}\), the sequences converge to it regardless of the choice of the perturbation. 

\begin{remark}
	Note that although the original problem \eqref{eq:circle SIP} has a unique optimizer, the relaxed problems, being linear programs, may not exhibit uniqueness of minimizers. This necessitates the use of regularization to extract the optimizer of \eqref{eq:circle SIP}
\end{remark}

The finitely constrained inner optimization problem was solved using \textsf{SLSQP} method in the \textsf{SciPy} library by providing the initial guess \((1, 0)\). The global optimization was solved using \texttt{dual\_annealing} method in the \textsf{SciPy} library coupled with other default parameters of the routine.

\subsection{Chebyshev centers under the \(\ell_{1}\) norm}

Let \(\R^{2}\) be equipped with the \(\ell_{1}\)-norm
\[
    \R^{2} \ni (\xpos, \ypos) \mapsto \norm{(\xpos, \ypos)}_{\ell_{1}} \Let \abs{\xpos} + \abs{\ypos} \in \R.
\]
Consider the set \(\relSet \subset \R^{2}\) defined by
\begin{align*}
	\relSet = \set[\bigg]{(\xpos, \ypos) \in \R^{2} \suchthat \begin{aligned}
				& \xpos + 2 \ypos \le 2,\;  -\xpos + 2 \ypos \le 2, \\
				& -\xpos + 4 \ypos \ge -2,\; \xpos \ge -2.
	\end{aligned}}.
\end{align*}
(This set \(\relSet\) was introduced in \secref{sec:MSA algo}.) We are interested in finding a Chebyshev center of \(\relSet\) in \((\R^{2}, \norm{\cdot}_{\ell_{1}})\).  Since \(\ell_{1}\)-norm is not strictly convex, the Chebyshev center of a set cannot be obtained directly from the MSA algorithm but the approach via regularization can be employed, and to this end we pick the perturbation function
\[
	\R^{3} \ni (\slack, \xpos, \ypos) \mapsto \regularizer(\slack, \xpos, \ypos) = \tfrac{1}{2} \bigl((\slack - 2)^{2} + (\xpos - 2)^{2} + (\ypos - 2)^{2}\bigr).
\]

\begin{figure}[h!]
    \centering
    \setlength\figureheight{=0.3\textwidth}
    \setlength\figurewidth{=0.45\textwidth}
        \input{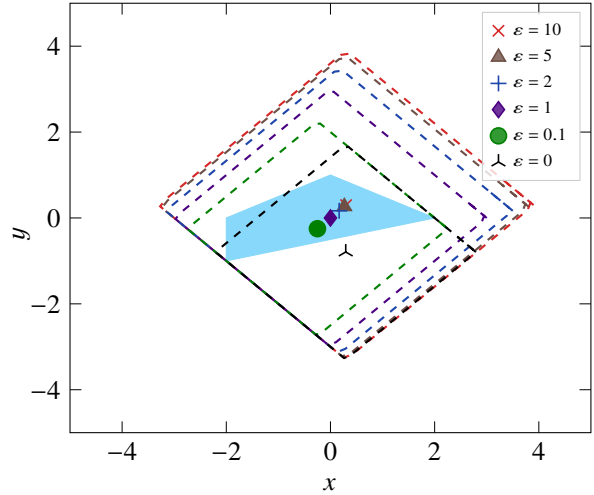}
	\caption{Sequence of solutions converging to a Chebyshev ball of \(\relSet\).}
	\label{fig:chebyshev example}
\end{figure}
Figure \ref{fig:chebyshev example} shows the \(\ell_{1}\)-balls obtained as solutions to the perturbed optimization problem for various \(\eps\). Observe that the radii of the candidate Chebyshev balls approach the radius \(2.5\) of the ball corresponding to \(\eps = 0\), which is the Chebyshev radius of \(\relSet\). But note that the ball corresponding to \(\eps = 0\) does not encompass all the points in \(\relSet\). The sequence of balls obtained via regularization satisfy all the constraints and their centers converge to the point \((-0.25, -0.25)\). Thus \((-0.25, -0.25)\) is a Chebyshev center of the set \(\relSet\).

It is important to note that solution picked by the MSA algorithm (without perturbation) may very well be a Chebyshev center of \(\relSet\) but this cannot be guaranteed in general. The approach of regularization guarantees that the sequence of solutions lies in the feasible set and hence also the limit.

\begin{remark}
    The global optimization routine plays a crucial role in this algorithm since its convergence to a global optimizer is imperative to establish the equivalence between the resulting finite convex minimization problem and the original SIP, and for the extraction of both the optimal value and an optimizer of the SIP.
\end{remark}

The inner optimization was solved using \textsf{SLSQP} algorithm of \textsf{SciPy} library and the global optimization problem was solved using the \texttt{differential\_evolution} method \cite{ref:Storn-97} of the SciPy library with the `\texttt{randtobest1exp}' option for the strategy parameter.

\subsection{A case of non-convex \(\relSet\)}
\label{s:non-convex}

	Here is a relatively simple example of solving the Chebyshev center problem for the non-convex region
	\[
		\relSet \Let \set[\big]{(x_1, x_2)\in\lcrc{0}{1}^2 \suchthat x_1^2 + x_2^2 \ge \tfrac{1}{3} \text{ and }(x_1-1)^2 + x_2^2 \ge \tfrac{2}{3}}
	\]
	relative to the standard Euclidean norm on \(\R^2\). The numerical calculations corresponding to the MSA algorithm were carried out in \textsf{Mathematica 12.1} using its native \texttt{NelderMead} technique in the global optimization routine \texttt{NMaximize}, and led to the Chebyshev radius \(\chebRadius{\relSet} = 0.633431\) and the Chebyshev center \((0.500000, 0.611112)\).\footnote{For this example we report numerical results correct up to $6$ decimal places.} A pictorial representation of the underlying set \(\relSet\) (shaded in blue) and the Chebyshev ball (shaded in light brown) is shown below:
	\[
		\includegraphics[scale=0.65]{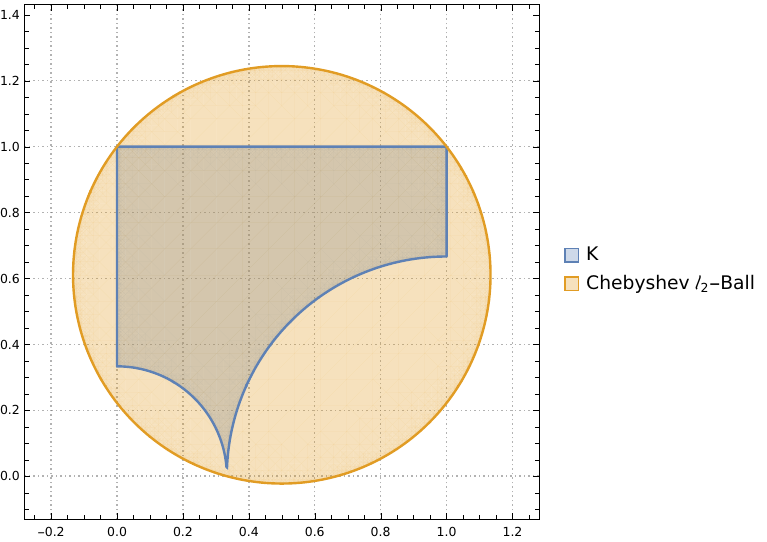}
	\]

\subsection{Examples on function spaces}

Consider an example of data fitting on the space \(\Lp{2}(\lcrc{0}{1}, \R)\). Let the hypothesis class \(\modelClass\) be the space of polynomials with bounded coefficients and the reduced model class be the polynomials with degree below \(\discDim\). The set of polynomials attaining prescribed values at specific datapoints is a restriction of the model class onto an affine space (in the coefficients).

\subsubsection{Simple 1D example}

We begin with the case of the affine space being 1-dimensional. In this scenario, a sample point can be represented by a single parameter, say \(\alpha\), and the induced norm on the affine space becomes equivalent to \(\abs{\cdot}\) in the \(\alpha\)- space. As a consequence, although the Chebyshev radius of \(\relSet\) varies, the Chebyshev center of \(\relSet\) becomes independent of the norm chosen on the norm chosen on the original space.

We consider an example in \(\R^{2}\). Let the affine space described by the data set be the line \(y = 2x + 3\). Suppose that \(\relSet\) is the intersection of the line with the rectangular region described by \(\abs{x} \leq 5, \abs{y} \leq 10\).  Figure \ref{fig:Cheb1d rect} shows the Chebyshev circles of \(\relSet\) under different norms (induced by different inner products) on \(\R^{2}\).

\begin{figure}[h!]
    \centering
    \setlength\figureheight{=0.3\textwidth}
    \setlength\figurewidth{=0.45\textwidth}
        \includegraphics[scale=0.23]{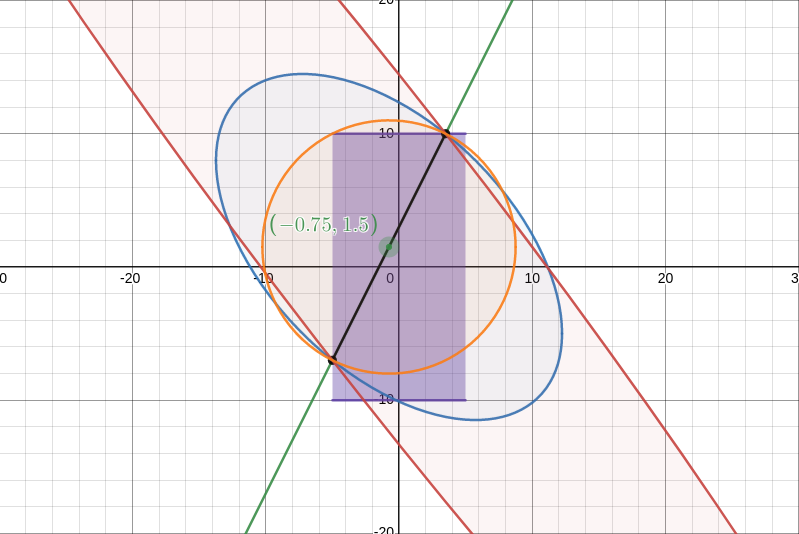}
    \caption{Chebyshev center of \(\relSet\) with rectangular bounding region .}
	\label{fig:Cheb1d rect}
\end{figure}

Figure \ref{fig:Cheb1d ell} shows the Chebyshev balls of \(\relSet\) which is now considered to be the portion of the line \(y = 2x + 3\) inside the elliptical region \(x^2 + 3y^2 \leq 100\).
\begin{figure}[h!]
    \centering
    \setlength\figureheight{=0.3\textwidth}
    \setlength\figurewidth{=0.45\textwidth}
        \includegraphics[scale=0.3]{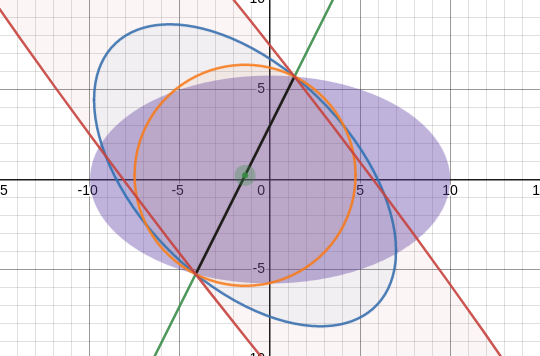}
    \caption{Chebyshev center of \(\relSet\) with elliptical bounding region.}
	\label{fig:Cheb1d ell}
\end{figure}

\subsubsection{Example of affine space in 3D}

Consider an example of a 2 dimensional affine space.

Concretely, let the affine space be the plane with the normal \(\pmat{1 & 0.7 & 0.49}\) passing through the point \(\pmat{1, 1, 1}\). Let \(\relSet\) be the region in the plane inside a cube of radius 15 centered at origin. Parametrising the affine plane using an orthonormal basis and centering at \(\pmat{1, 1, 1}\), the region \(\relSet\) is depicted as the purple shaded region in Figure \ref{fig:Cheb2D}.
\begin{figure}[h!]
    \centering
    \setlength\figureheight{=0.3\textwidth}
    \setlength\figurewidth{=0.45\textwidth}
        \includegraphics[scale=0.5]{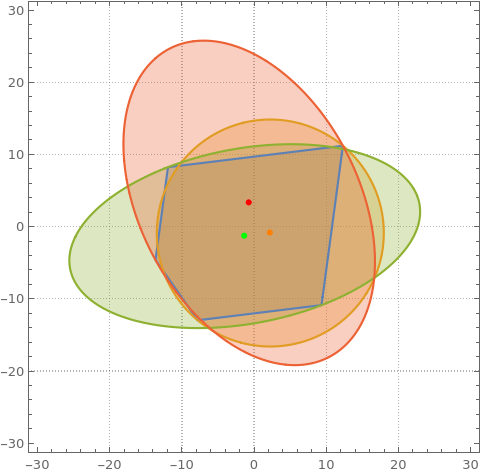}
    \caption{The purple shaded region is the set \(\relSet\) and the orange circle represents the boundary of the Chebyshev ball corresponding to the Euclidean norm.}
	\label{fig:Cheb2D}
\end{figure}

The orange shaded region in Figure \ref{fig:Cheb2D} denotes the Chebyshev ball obtained when the space is equipped with the Euclidean norm.
Observe that, in this case, the Chebyshev center depends on the norm used on the ambient space.

\subsubsection{Higher dimensional space}

We now consider an example in the space of polynomials. Let the search space \(\discSpace \subset \Lp{2} (\lcrc{0}{1}, \R)\) also be the polynomials with degree below \(20\). The data \(\dataSet\) is obtained by sampling the function
\[
    \lcrc{0}{1} \ni x \mapsto f(x) \Let 10 \sin(2 \pi x) \in \R
\]
at \(\sampleSize\) points in the interval \(\lcrc{0}{1}\).

Using the basis \(\set{1, x, x^2, \ldots, x^{\discDim-1}}\), the set of functions in \(\modelClass\) satisfying the data \(\dataSet\) can be seen to be the intersection of an affine space with \(\modelClass\):
\[
    \relSet = \set[\bigg]{\conCoeffs[] \in \R^{\discDim} \suchthat \sum_{j=0}^{\discDim-1} \conCoeffs x_{k}^{j} = f(x_{k}) \quad \text{for k} = 1, \ldots, \sampleSize}
    \cap \set[\big]{\conCoeffs[] \in \R^{\discDim} \suchthat \abs{\conCoeffs} \le 100}
\]

Using the same basis for the search space \(\discSpace\), the relative Chebyshev center of the set \(\relSet\) in \(\discSpace\) can be phrased as the optimizaiton problem:
\[
    \argmin \limits_{\funcCoeffs[] \in \R^{21}} \max \limits_{\conCoeffs[] \in \relSet} \norm{\funcCoeffs[] - \conCoeffs[]}_{\modelClass},
\]
where \(\norm{\cdot}_{\modelClass}\) is the induced norm on \(\discSpace\) identified as \(\R^{\discDim}\).

The effective dimension of the discrete model class is the dimension of the affine space, \(\discDim - \sampleSize\). Since, in normed spaces, the Chebyshev center of an affine space lies in the same affine space, the Chebyshev center of \(\relSet\) also satisfies the prescribed data.

The figures \ref{fig:design 7} -- \ref{fig:design 20} showcase the Chebyshev centers of \(\relSet\) for a series of sampling size \(\sampleSize\) and reduced search dimension \(\discDim\).

\begin{figure}[h!]
    \centering
    \setlength\figureheight{=0.3\textwidth}
    \setlength\figurewidth{=0.45\textwidth}
        \includegraphics[scale=0.5]{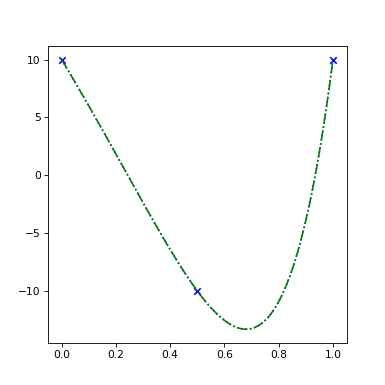}
    \caption{Chebyshev center for \(\sampleSize = 3\) sampling points with \(\discDim = 7\).}
	\label{fig:design 7}
\end{figure}

\begin{figure}[h!]
    \centering
    \setlength\figureheight{=0.3\textwidth}
    \setlength\figurewidth{=0.45\textwidth}
        \includegraphics[scale=0.5]{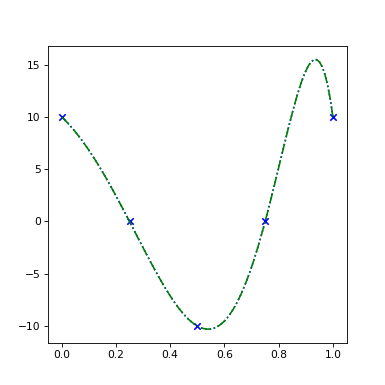}
    \caption{Chebyshev center for \(\sampleSize = 5\) sampling points with \(\discDim = 10\).}
	\label{fig:design 10}
\end{figure}

\begin{figure}[h!]
    \centering
    \setlength\figureheight{=0.3\textwidth}
    \setlength\figurewidth{=0.45\textwidth}
        \includegraphics[scale=0.5]{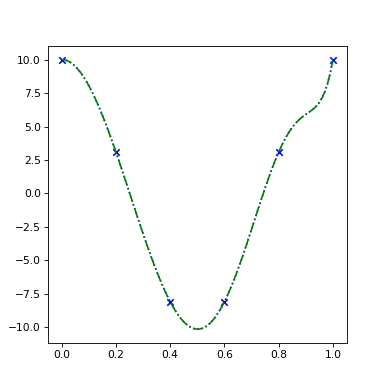}
    \caption{Chebyshev center for \(\sampleSize = 6\) sampling points with \(\discDim = 12\).}
	\label{fig:design 12}
\end{figure}

\begin{figure}[h!]
    \centering
    \setlength\figureheight{=0.3\textwidth}
    \setlength\figurewidth{=0.45\textwidth}
        \includegraphics[scale=0.5]{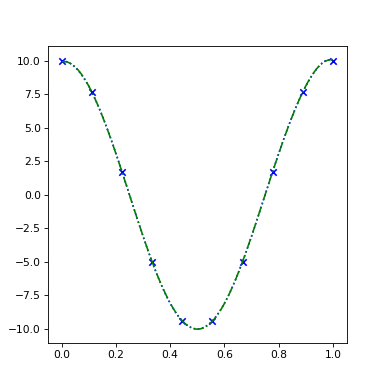}
    \caption{Chebyshev center for \(\sampleSize = 10\) sampling points with \(\discDim = 20\).}
	\label{fig:design 20}
\end{figure}

Figure \ref{fig:design 10q} presents a variation of the case showed in Figure \ref{fig:design 10} wherein the set \(\relSet\) is modified to be the intersection of the affine space satisfying the prescribed data with a shifted bounding box on the coefficients:
\[
    \relSet' = \set[\bigg]{\conCoeffs[] \in \R^{\discDim} \suchthat \sum_{j=0}^{\discDim-1} \conCoeffs x_{k}^{j} = f(x_{k}) \quad \text{for k} = 1, \ldots, \sampleSize} \cap \set[\big]{\conCoeffs[] \in \R^{\discDim} \suchthat 100 \le \conCoeffs \le 200}.
\]
It is observed that the change in the Chebyshev center due to this modification to the set \(\relSet\) is a translation of the coefficients by an amount related to the change in the bounds.
\begin{figure}[h!]
    \centering
    \setlength\figureheight{=0.3\textwidth}
    \setlength\figurewidth{=0.45\textwidth}
        \includegraphics[scale=0.5]{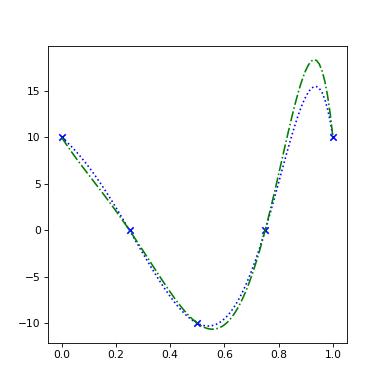}
    \caption{Chebyshev center for \(\sampleSize = 5\) sampling points with \(\discDim = 10\) with shifted bounding box on the coefficients.}
	\label{fig:design 10q}
\end{figure}


\subsection{Critical importance of the global optimization}

The MSA algorithm extracts the Chebyshev center by relying on a global optimization process. The convergence of the global optimization process to an actual global optimum is crucial to finding the actual Chebyshev center of a set. Consequently, the global optimization algorithm and its sampling process should ideally be fine-tuned depending on the application at hand.

Let us illustrate the gaps in the capabilities of off-the-shelf solvers in the context of some of preceding examples:
\begin{itemize}[label=\(\circ\), leftmargin=*]
	\item The Chebyshev triplet obtained for the non-convex set \(\relSet\) in \secref{s:non-convex} in \textsf{Mathematica 12.1} using its native \texttt{SimulatedAnnealing} technique in the global optimization routine \texttt{NMaximize} led to the Chebyshev radius \(\chebRadius{\relSet} = 0.600925\) and the Chebyshev center \((0.499999, 0.666666)\) (correct up to \(6\) decimal places). A pictorial representation of the underlying set \(\relSet\) (shaded in blue) and this particular Chebyshev ball (shaded in light brown) is shown below for comparison against the figure reported in \secref{s:non-convex}.
		\[
			\includegraphics[scale=0.65]{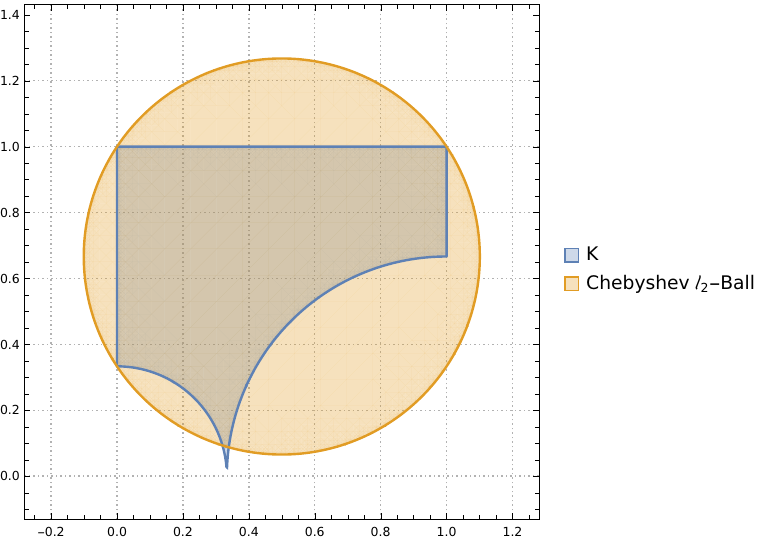}
		\]
		This difference between the two outputs is perhaps due to the difficulty faced by the native \texttt{SimulatedAnnealing} routine to sample from the pinched (which is nearly singular relative to the Lebesgue measure) region of \(\relSet\) close to \((\frac{1}{3}, 0)\), in contrast to the simplex-based deterministic \texttt{NelderMead} routine which led to the correct Chebyshev triplet in \secref{s:non-convex}.
    \item Consider the problem of finding the Chebyshev ball for the equilateral triangle \(\relSet = \set{(x, y) \in \R^{2} \suchthat  \sqrt{3} x + y \le \sqrt{3}, -\sqrt{3} x + y \le \sqrt{3}, y \ge 0}\) under the weighted norm \(\norm{v}_{M} \Let \sqrt{\inprod{v}{M v}}\) with weighting matrix
        \[
            M = \pmat{4.01933 & -2.038 \\ -2.038 & 14.6273}.
        \]
        The \texttt{NelderMead} routine in \textsf{Mathematica 12.1} converges to a suboptimal Chebyshev radius of \(3.706789\); the corresponding Chebyshev ball is shown as the green shaded region in the following figure. The correct Chebyshev radius obtained by including the vertices of the triangle in the constraints is \(3.709497\) and is shaded in orange in the figure below.
        \[
            \includegraphics[scale=0.65]{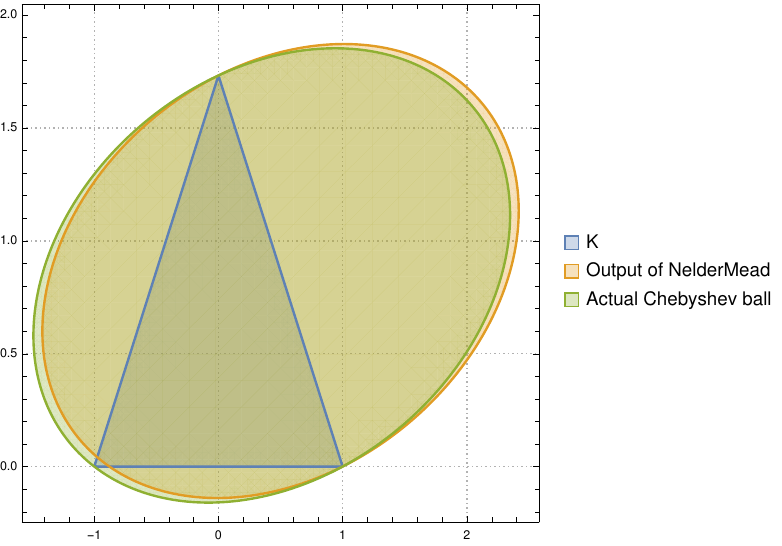}
        \]
\end{itemize}
These two illustrative examples highlight the need, in general, for careful context-dependent tuning of the global optimization algorithm in the context of the MSA algorithm.

\providecommand{\bysame}{\leavevmode\hbox to3em{\hrulefill}\thinspace}
\providecommand{\MR}{\relax\ifhmode\unskip\space\fi MR }
\providecommand{\MRhref}[2]{%
  \href{http://www.ams.org/mathscinet-getitem?mr=#1}{#2}
}
\providecommand{\href}[2]{#2}

\end{document}